\documentclass[10pt,reqno,a4paper]{amsart}

\usepackage[hmargin=2.6cm,bmargin=2.6cm,tmargin=3cm]{geometry}

\usepackage{amsmath}
\usepackage{amssymb} 
\usepackage{latexsym}
\usepackage{array}
\usepackage{ifthen}
\usepackage{amsthm}
\usepackage{mathrsfs}
\usepackage{rotating,pgf,tikz}
 
\author{James B.~Kennedy}
\author{Robin Lang}

\address{James B.~Kennedy, Grupo de F{\'i}sica Matem\'atica, Faculdade de Ci\^encias, Universidade de Lisboa, Campo Grande, Edif{\'i}cio C6, P-1749-016 Lisboa, Portugal}
\email{jbkennedy@fc.ul.pt}

\address{Robin Lang, Institut f\"ur Analysis, Dynamik und Modellierung, Universit\"at Stuttgart, Pfaffenwaldring 57, D-70569 Stuttgart, Germany}
\email{robin.lang@mathematik.uni-stuttgart.de}

\numberwithin{equation}{section}

\newtheorem{theorem}{Theorem}[section]
\newtheorem{corollary}[theorem]{Corollary}
\newtheorem{lemma}[theorem]{Lemma}

\newtheorem{conjecture}[theorem]{Conjecture}

\theoremstyle{definition}

\newtheorem{remark}[theorem]{Remark}

\newcommand{\R}{\mathbb{R}}
\newcommand{\N}{\mathbb{N}}
\newcommand{\C}{\mathbb{C}}

\newcommand{\V}{\mathcal{V}}
\newcommand{\G}{\mathcal{G}}
\newcommand{\A}{\mathcal{A}}
\newcommand{\E}{\mathcal{E}}

\newcommand{\Scal}{\mathcal{S}}
\renewcommand{\G}{\mathcal{G}}

\DeclareMathOperator{\dist}{dist}
\DeclareMathOperator{\diag}{diag}

\DeclareMathOperator{\GM}{\mathfrak{m}}

\renewcommand{\Im}{{\rm Im}\,}
\renewcommand{\Re}{{\rm Re}\,}
\newcommand{\I}{{\rm i}}
\newcommand{\e}{{\rm e}}

\newcommand{\MD}{\mathfrak{D}}

\newcommand{\dx}{\text{d}x}

\let\phi\varphi
\let\epsilon\varepsilon

\title[Quantum graph Laplacians with large complex $\delta$ couplings]{On the eigenvalues of quantum graph Laplacians with large complex $\delta$ couplings}



\thanks{The authors would like to thank Timo Weidl for originally suggesting the problem, and Sabine B\"ogli for many helpful discussions. The work of the authors was supported by the Funda{\c{c}}{\~a}o para a Ci{\^e}ncia e a Tecnologia, Portugal, via the program ``Investigador FCT'', reference IF/01461/2015 (J.B.K.), and project PTDC/MAT-CAL/4334/2014 (both authors). The work of Robin Lang was supported by the Deutsche Forschungsgemeinschaft (DFG) through the Research Training Group 1838: Spectral Theory and Dynamics of Quantum Systems.}

\begin{document}

\maketitle

\begin{abstract}
We study the location of the spectrum of the Laplacian on compact metric graphs with complex Robin-type vertex conditions, also known as $\delta$ conditions, on some or all of the graph vertices. We classify the eigenvalue asymptotics as the complex Robin parameter(s) diverge to $\infty$ in $\C$: for each vertex $v$ with a Robin parameter $\alpha \in \C$ for which $\Re\alpha \to -\infty$ sufficiently quickly, there exists exactly one divergent eigenvalue, which behaves like $-\alpha^2/\deg v^2$, while all other eigenvalues stay near the spectrum of the Laplacian with a Dirichlet condition at $v$; if $\Re \alpha$ remains bounded from below, then all eigenvalues stay near the Dirichlet spectrum. Our proof is based on an analysis of the corresponding Dirichlet-to-Neumann matrices (Titchmarsh--Weyl $M$-functions). We also use sharp trace-type inequalities to  prove estimates on the numerical range and hence on the spectrum of the operator, which allow us to control both the real and imaginary parts of the eigenvalues in terms of the real and imaginary parts of the Robin parameter(s).
\end{abstract}

\section{Introduction}
\label{sec:introduction}

Consider the eigenvalue problem for the Laplacian with Robin boundary conditions
\begin{equation}
\label{eq:robin-laplacian}
\begin{aligned}
	-\Delta u &= \lambda u \qquad &&\text{in } \Omega,\\
	\frac{\partial u}{\partial\nu} + \alpha u &=0 \qquad &&\text{on } \partial\Omega,
\end{aligned}
\end{equation}
on a bounded Lipschitz domain $\Omega\subset\R^d$, $d \geq 2$ (or a bounded interval if $d=1$), where $\nu$ is the outer unit normal to $\Omega$, and $\alpha$, which is most commonly taken as a real number but may be complex or a function defined on the boundary $\partial\Omega$, is to be thought of as a parameter. If $\alpha \in\R$, then \eqref{eq:robin-laplacian} admits a sequence of real eigenvalues, which we number by increasing size and repeat according to their finite multiplicities, $\lambda_1 (\alpha) \leq \lambda_2 (\alpha) \leq \ldots \to \infty$; these eigenvalues are piecewise analytic functions of $\alpha$.

The asymptotic behaviour of these eigenvalues as the Robin parameter $\alpha \in \R$ becomes large has been studied intensively over the last decade, in particular in the singular limit $\alpha \to -\infty$; we refer to the recent survey \cite{BFK} as well as, e.g., \cite{BruneauPopoff,ExnerMinakovParnovski,FreitasKrejcirik,HelfferKachmar,Khalile,KovarikPankrashkin,PankrashkinPopoff} and the references therein. Briefly, if $\alpha \to +\infty$, then $\lambda_k (\alpha)$ converges from below to the $k$th eigenvalue of the Laplacian with Dirichlet (zero) boundary conditions, while if $\alpha \to -\infty$, then there exists a sequence of eigenvalues each of which diverges like $-C\alpha^2$ for some constant $C \geq 1$ (which is not yet fully understood and may depend on $\Omega$ and the eigenvalue curve in question), while any bounded analytic curve of eigenvalues converges to some eigenvalue of the Dirichlet Laplacian from above; we refer to \cite{BKL,BFK} for more details.

The question as to what happens when $\alpha$ is a large \emph{complex} parameter, corresponding to an impedance boundary condition, was recently asked for the first time in \cite{BKL}. Although the problem is a natural extension of the real case, the operator on $L^2(\Omega)$ associated with the problem \eqref{eq:robin-laplacian}, that is, the Robin Laplacian, is clearly no longer self-adjoint if $\alpha \not\in \R$. Thus, although the eigenvalues of the operator still form analytic families in dependence on $\alpha$ whose asymptotic behaviour should be compatible with the asymptotic behaviour for real $\alpha$, all the variational techniques used in the above-mentioned works, such as the variational characterisation of the eigenvalues and Dirichlet-Neumann bracketing, are no longer applicable, and new methods are needed. We start by recalling the conjecture made in \cite[Conjecture~1.2]{BKL} on the expected behaviour on general domains.

\begin{conjecture}
\label{conj:general}
Let $\Omega \subset \R^d$, $d\geq 2$, be a bounded Lipschitz domain, and suppose that $\alpha \to \infty$ in $\C$.
\begin{enumerate}
\item If $\Re \alpha \to -\infty$, then there exists an infinite family of analytic branches of absolutely divergent eigenvalues behaving like
\begin{enumerate}
\item[(i)] $-\alpha^2 + o(\alpha^2)$ if $\partial\Omega$ is of class $C^1$;
\item[(ii)] $-C\alpha^2 + o(\alpha^2)$, where the constant $C\geq 1$ may depend on $\Omega$ and the corresponding curve of eigenvalues otherwise.
\end{enumerate}
Every other eigenvalue converges to an eigenvalue of the Dirichlet Laplacian.
\item If $\Re \alpha$ remains bounded from below, then every eigenvalue converges to an eigenvalue of the Dirichlet Laplacian.
\end{enumerate}
\end{conjecture}

This was proved in \cite{BKL} in the special cases where $\Omega$ is an interval, a $d$-dimensional rectangle and a ball using the duality between the eigenvalues of \eqref{eq:robin-laplacian} and those of appropriate Dirichlet-to-Neumann operators.

Our principal goal here is to study the corresponding problem in the setting of \emph{quantum graphs}, that is, metric graphs on which a differential operator, here the Laplacian, acts, and in particular lend weight to Conjecture~\ref{conj:general} by proving a version of it in this setting. More precisely, we consider compact metric graphs $\G = (\V,\E)$ consisting of a finite set of edges $\E$, each identified with a compact interval, joined together in a certain way at a finite set of vertices $\V$. We then define a differential operator on $\G$ as follows: on each edge we take the negative of the second derivative of the function (i.e., the Laplacian), while we impose certain vertex conditions at the vertices as analogues of the boundary conditions on a domain. We will return to this briefly in a moment (for more details, see also Section~\ref{sec:robin-on-quantum-graphs}, or, e.g., \cite[Chapter~1]{BerkolaikoKuchment}).

Quantum graphs are extremely useful models in spectral theory, as on the one hand such graph Laplacians tend to display complex behaviour characteristic of higher-dimensional Laplacians on domains or manifolds, while on the other hand they are essentially one-dimensional objects and thus more amenable to detailed analyses and even explicit computations, while still being highly non-trivial. This is the case for problems such as Anderson localisation, quantum chaos, the Bethe--Sommerfeld conjecture, or also geometric spectral theory.  However, they also appear frequently as models of a number of phenomena in their own right, in particular the propagation of waves in thin networks at very small scales, e.g., in thin waveguides, quantum wires and carbon nano-structures, among many others. For more information on, and references to, all these topics in the context of quantum graphs, we refer to \cite[Preface and Chapters~1 and ~7]{BerkolaikoKuchment}, as well as, e.g., \cite{BKKM,ExnerTurek,KottosSmilansky}.

In our case, we will study quantum graph Laplacians equipped with a $\delta$-type condition at some or all of the vertices of the graph. This vertex condition, also known as a $\delta$ coupling or $\delta$ interaction, is considered the natural analogue of Robin boundary conditions on domains, and appears frequently in the quantum graph literature (see, e.g., \cite[Section~1.4]{BerkolaikoKuchment} for a description; these conditions featured prominently in \cite{BKKM,BLS,ExnerJex12,ExnerTurek,Hussein,HKS,RiviereRoyer}, among many others). More precisely, we will assume that the functions $f$ in the domain of our operator satisfy 
\begin{enumerate}
\item[(i)] continuity at all vertices $v \in \V$,
\item[(ii)] the $\delta$ condition
\begin{displaymath}
	\sum_{e \sim v_j} \frac{\partial}{\partial\nu} f|_e (v_j) + \alpha_j f(v_j) = 0,
\end{displaymath}
$\alpha_j \in \C$, $j=1,\ldots,k$, at a distinguished set $\V_R := \{v_1,\ldots,v_k\} \subset \V$ of Robin vertices (here $f|_e$ is the restriction of the function $f$ on $\G$ to the edge $e$, $\frac{\partial}{\partial\nu} f|_e (v)$ is the derivative of $f$ at the endpoint of $e$ pointing into $v_j$, and the summation is over all edges $e$ incident with $v_j$), and
\item[(iii)]  the usual Kirchhoff condition (also known as current conservation, see \cite[eq.~(1.4.4)]{BerkolaikoKuchment}), corresponding to $\alpha = 0$, at all vertices in $\V \setminus \V_R$.
\end{enumerate}
For brevity, we will write $-\Delta_{\V_R}^\alpha$ for the corresponding Robin Laplacian defined on $L^2(\G)$, where $\alpha$ is shorthand for the vector $(\alpha_1,\ldots,\alpha_k) \in \C^k$, and note that all eigenvalues of $-\Delta_{\V_R}^\alpha$, which form a countable set for each $\alpha$, are at least piecewise analytic functions of $\alpha \in \C^k$; again, see Section~\ref{sec:robin-on-quantum-graphs} for more details on all of this.

Our main result is a version of Conjecture~\ref{conj:general} for such quantum graphs with $\delta$ conditions: for each vertex $v_j$ for which $\Re \alpha_j \to -\infty$ sufficiently quickly, we obtain a single divergent eigenvalue, corresponding to the principle that this is a one-dimensional perturbation in a certain sense; while if $\Re \alpha_j$ remains bounded from below as $\alpha_j \to \C$, then in the limit we end up with a Dirichlet (zero) condition at the vertex $v_j$. We will denote by $-\Delta_{\V_R}^D$ the corresponding Laplacian which has such a Dirichlet condition at all vertices in $\V_R$ and continuity plus Kirchhoff conditions at $\V \setminus \V_R$. 

\begin{theorem}
\label{thm:qg-behaviour}
Suppose $\G = (\V,\E)$ is a compact metric graph, and for the set of Robin vertices $\V_R = \{v_1,\ldots, v_k \}\subset \V$, suppose that each $v_j \in\V_R$ is equipped with the Robin parameter $\alpha_j \in \mathbb{C}$, $j=1,\ldots,k$, and set $\alpha := (\alpha_1,\ldots,\alpha_k) \in \C^k$. We suppose that for some $m \in \{0,1,\ldots,k\}$
\begin{enumerate}
\item $\alpha_j \to\infty$ in a sector fully contained in the open left half-plane, for all $1\leq j\leq m$;
\item $\alpha_j \to\infty$ in such a way that $\Re\alpha_j$ remains bounded from below as $\alpha_j\to\infty$, for all $m+1\leq j\leq k$.
\end{enumerate}
Then, as $\alpha \to \infty$, counting multiplicities there are exactly $m$ eigenvalues $\lambda$ of $-\Delta_{\V_R}^\alpha$ which diverge away from the positive real semi-axis (that is, whose distance to the positive real semi-axis grows to $\infty$); these satisfy the asymptotics 
\begin{equation}
\label{eq:qg-divergent-ev}
	\lambda = - \frac{\alpha_j^2}{(\deg v_j)^2} + \mathcal{O}\left(\alpha_j^2\e^{\ell_\G\Re\alpha_j}\right)
\end{equation}
as $\alpha \to \infty$, where $\ell_\G$ is the length of the shortest edge of $\G$. Every eigenvalue of $-\Delta_{\V_R}^\alpha$ which does not diverge to $\infty$ in $\C$ converges to an eigenvalue of $-\Delta_{\V_R}^D$.
\end{theorem}

Before proceeding, a couple of observations are in order. Firstly, it is natural to ask what happens if the non-Robin vertices are equipped with some other self-adjoint condition(s) than the standard continuity-Kirchhoff ones; while our proof is set up to work only for the latter, we expect certain generalisations would be possible (see Remark~\ref{rem:kirchhoff-in-VN} for more details). Secondly, in the statement of the theorem, we deliberately avoid considering any potential eigenvalues diverging within finite distance of the positive real semi-axis, where the eigenvalues of $-\Delta_{\V_R}^D$ are located and the relationship between $\alpha$ and $\lambda$ is far more complicated (cf.\ \cite[Section~9.1.4]{BKL} for a discussion of what happens in the much simpler but already involved case of the interval); it would go beyond the scope of this note to classify all possible types of behaviour in this case.

In addition to supporting Conjecture~\ref{conj:general}, Theorem~\ref{thm:qg-behaviour} should be of independent interest for quantum graphs, and indeed this serves as a second motivation: as mentioned, such $\delta$ vertex couplings, usually real but sometimes complex, arise frequently in the spectral theory of quantum graphs, where it is useful to understand the behaviour of the eigenvalues and eigenfunctions as functions of $\alpha$ (as used extensively in \cite{BKKM,ExnerJex12}, for example). Basic spectral and generation properties of graph Laplacians with complex $\delta$ conditions in particular were treated extensively in \cite{Hussein,HKS}, and just recently a Weyl law for the asymptotics of the large eigenvalues of star graphs for fixed complex $\alpha$ was established in \cite{RiviereRoyer}.

In fact, Theorem~\ref{thm:qg-behaviour} also represents a certain generalisation of Conjecture~\ref{conj:general} in the sense that $\alpha$ is variable, i.e., may depend on the vertex. This includes as a special case an important prototype model for $\mathscr{P}\!\mathscr{T}$-symmetry originally introduced in \cite{KBZ} and since studied by many authors, often in the context of thin waveguides or layers (see, e.g., \cite{BorisovKrejcirik2,BorisovKrejcirik1,BorisovZnoiil,LotoreichikSiegl,Novak} and the references therein). In \cite{KBZ} the authors consider the Laplacian on a finite interval $(0,d)$ equipped (in our notation) with the Robin condition $-\I t$ at $0$ and $+\I t$ at $d$; in this case, it is possible to calculate the spectrum explicitly, and it turns out that the eigenvalues are exactly the Dirichlet eigenvalues plus an additional eigenvalue $t^2$, diverging along the positive semi-axis (see \cite[Section~3]{KBZ}).

Finally, we note that there is a huge literature on the problem of eigenvalue asymptotics for Laplacians on $\R^d$ subject to $\delta$ (or even so-called $\delta'$) interactions supported on lower-dimensional manifolds of Euclidean space as the interaction becomes strong, closely related to the Robin eigenvalue problem on domains discussed above (e.g., \cite{DEKP,ExnerJex13,EKL} and the references therein); such $\alpha$ are often used to model potentials supported on a lower-dimensional manifold (whence the alternative name ``$\delta$ potential'' for them). And yet on metric graphs, even in the case of real $\alpha$ to the best of our knowledge the asymptotic behaviour in $\alpha$ described by \eqref{eq:qg-divergent-ev} is new. We thus state this case explicitly for the record.

\begin{theorem}
\label{cor:qg-behaviour-neg-alpha}
Keep the assumptions of Theorem~\ref{thm:qg-behaviour}. Suppose now that $\alpha:=\alpha_1=\ldots=\alpha_k$ is real and negative and all vertices in $\V_R$ are equipped with the common Robin parameter $\alpha$, and that $\deg v_1 \leq \deg v_2 \leq \ldots \leq \deg v_k$. Then for $\alpha < -2\max\limits_{j=1,\dots,k}\left\{\frac{\deg v_j}{\ell_j}\right\}$ the self-adjoint operator $-\Delta_{\V_R}^\alpha$ has exactly $k$ negative eigenvalues (here $\ell_j$ is the length of the shortest edge incident with $v_j$). Moreover, for each $j=1,\ldots,k$, the $j$th eigenvalue $\lambda_j = \lambda_j (\alpha)$ behaves like
\begin{equation}
\label{eq:qg-divergent-ev-neg-alpha}
	\lambda_j(\alpha) = - \frac{\alpha^2}{(\deg v_j)^2} + \mathcal{O}\left(\alpha^2\e^{\ell_\G\alpha}\right)
\end{equation}
as $\alpha \to -\infty$. Every other eigenvalue $\lambda_j(\alpha)$, $j \geq k+1$, converges to an eigenvalue of $-\Delta_{\V_R}^D$.
\end{theorem}

We believe the fact that the remaining (non-divergent) eigenvalues converge to the Dirichlet spectrum to be reasonably well known in the real case; in fact, one can show that $\lambda_j (\alpha)$ converges to the $(j-k)$th eigenvalue of $-\Delta_{\V_R}^D$ from above, for any $j\geq k+1$ (see \cite[Theorem~3.1.13]{BerkolaikoKuchment} for the proof when $k=1$; the general case is analogous). We include a short, direct proof of Theorem~\ref{cor:qg-behaviour-neg-alpha}, including the explicit estimate on $\alpha$, for the sake of completeness and concreteness, although the key point here is the asymptotic behaviour of the divergent eigenvalues, a direct corollary of Theorem~\ref{thm:qg-behaviour}. In this vein we draw explicit attention to the presence of the degree of the vertex in the asymptotics: the presence of a coefficient $C<1$ in the leading term asymptotics $-C\alpha^2$ appears to be new, and is at any rate in marked contrast to the known behaviour on domains in $\R^d$. There, in the smooth case the divergent eigenvalues behave like $-\alpha^2$ (as proved in \cite{DanersKennedy,LouZhu}), while the presence of corners at the boundary causes the appearance of eigenvalues behaving like $-C\alpha^2$ for some $C>1$ (as first observed in \cite{LOS} and studied extensively in \cite{Khalile,KhalilePankrashkin,LevitinParnovski}).

In this spirit we will also provide a number of estimates on the location of the eigenvalues, in fact the numerical range of the Robin Laplacian, which for any fixed parameter $\alpha \in \C^k$ bound them within a certain parabolic region of $\C$; in particular, this gives us control over both the real and imaginary parts of any eigenvalues in terms of the real and imaginary parts of $\alpha$ (see Theorems~\ref{thm:numerical-range} and~\ref{thm:imaginary-control} and Corollary~\ref{cor:est-real}). Moreover, these bounds are essentially asymptotically optimal when $\alpha$ has large negative real part; for example, if $\alpha \in (-\infty,0)$ is independent of the $k \geq 1$ Robin vertices, then (keeping the notation and setup from Theorem~\ref{cor:qg-behaviour-neg-alpha}) we obtain the following two-sided bound on the lowest eigenvalue $\lambda_1 (\alpha)$,
\begin{equation}
\label{eq:est-real-two-sided}
	-\frac{\alpha^2}{(\deg v_1)^2} + \frac{\alpha}{\ell_\G \deg v_1} \leq \lambda_1 (\alpha) 
	< \min \left\{-\frac{\alpha^2}{(\deg v_1)^2} - \frac{2\alpha}{\ell_\G \deg v_1} - \frac{1}{\ell_\G^2}, \frac{k\alpha}{|\G|} \right\}
\end{equation}
where we recall that $\deg v_1$ is the minimal degree of all vertices in the Robin vertex set $\V_R$ and $\ell_\G>0$ is the length of the shortest edge in $\G$, and $|\G|$ is the total length of $\G$ (the sum of all edge lengths); for the proof of \eqref{eq:est-real-two-sided} and more details see Corollary~\ref{cor:est-real} and Remark~\ref{rem:est-real}.

The proof of Theorem~\ref{thm:qg-behaviour} is based on the duality between eigenvalues of the Robin Laplacian and those of Dirichlet-to-Neumann-type operators, or more precisely matrices (see Theorem~\ref{thm:bk-duality}), also known as (Titchmarsh--Weyl) $M$-functions; the key to the proof is a well-chosen representation of the latter matrices, which allows a description of their asymptotics as functions of the relevant spectral parameter $\lambda$. In fact, one of the advantages of quantum graphs is that unlike in the case of domains it is possible to derive such more or less explicit formulae for these matrices, and this is what will allow us to give an essentially complete answer to the question of the behaviour of the eigenvalues in the presence of large complex Robin parameters.

This paper is organised as follows. In the preliminary Section~\ref{sec:robin-on-quantum-graphs} we give a brief summary of, and our notation for, metric graphs; and we then introduce the Robin and Dirichlet Laplacians on graphs, the operators with which we will be working. Sections~\ref{sec:dno} and~\ref{sec:dno-asymptotics} are devoted to the proof of Theorem~\ref{thm:qg-behaviour}. We start out in Section~\ref{sec:dno} by introducing Dirichlet-to-Neumann matrices and deriving the representation of them (Lemma~\ref{lem:dno-qg-representation}) that will then allow us to determine their asymptotic behaviour and hence prove Theorem~\ref{thm:qg-behaviour} in Section~\ref{sec:dno-asymptotics}. In Section~\ref{sec:numerical-range} we give the aforementioned estimates on the numerical range of Robin Laplacian and the real and imaginary parts of its eigenvalues; the corresponding proofs are the subject of Section~\ref{sec:numerical-range-proofs}, where we also give the proof of Theorem~\ref{cor:qg-behaviour-neg-alpha}.

\section{The Robin problem on compact quantum graphs}
\label{sec:robin-on-quantum-graphs}

\subsection{On quantum graphs}
\label{subsec:quantum-graphs}
We first need to recall some basic terminology; we refer to the monographs \cite{BerkolaikoKuchment,Mugnolo} or the elementary introduction \cite{Berkolaiko} for more details. A \emph{compact metric graph} $\G = (\V,\mathcal{E})$ consists of a finite vertex set $\V = \{v_1,\ldots,v_n\}$ and a finite edge set $\E = \{e_1,\ldots,e_m\}$, where each edge $e$ is identified with a compact interval $[0,\ell_e]$ of length $\ell_e>0$, denoted by $e\simeq [0,\ell_e]$, and where the endpoints $0$ and $\ell_e$ correspond to the vertices which are incident with the edge $e$. While this implicitly presupposes an orientation on $e$, it is a standard result that up to unitary equivalence the differential operators we will be considering do not depend on this choice of orientation.

We write $e \sim v$ to mean that the vertex $v$ is incident with the edge $e$. The degree of a vertex $v\in\V$, denoted by $\deg v\geq 1$, is the number of edges with which $v$ is incident. We explicitly allow our graphs to have loops (edges both of whose endpoints correspond to the same vertex; in this case the edge is counted twice when computing the degree of the vertex) and we allow multiple edges between any given pair of vertices. Equipped with the usual metric corresponding to the shortest Euclidean path between two points, $\G$ is a compact metric space. The graph is connected if and only if it is connected as a metric space. We will \emph{always} assume $\G$ to be such a connected compact metric graph.

On $\G$, as customary, we can define the space $L^2(\G)$ of square integrable functions, the space $C(\G) \hookrightarrow L^2 (\G)$ of continuous functions and the Sobolev space $H^1 (\G) \hookrightarrow C(\G)$, respectively, by
\begin{displaymath}
\begin{aligned}
	L^2 (\G) &= \bigoplus_{e \in \mathcal{E}} L^2 (e) \simeq \bigoplus_{e \in \mathcal{E}} L^2 ((0,\ell_e)),\\
	C(\G) &= \{ f: \G \to \C \;:\; f|_e \in C(e) \text{ for all } e \in \mathcal{E}\\
	& \qquad\qquad \text{ and $f$ is continuous at each } v \in \V \},\\
	H^1(\G) &= \{f \in C(\G)\;:\; f|_e \in H^1(e) \text{ for all } e \in \mathcal{E} \};
\end{aligned}
\end{displaymath}

we also write
\begin{equation}
\label{eq:integral-graph}
	\int_{\G} f\,\dx := \sum_{e \in \mathcal{E}} \int_e f|_e \,\dx
\end{equation}
for the integral of a function $f$ over $\G$, as well as $\frac{\partial}{\partial\nu} f|_e (v)$ for the derivative of $f$ along the edge $e$ at $v$, pointing into $v$ (which may be thought of as the outer normal derivative to the edge $e$ at $v$); this exists if $f|_e \in C^1 (e)$.

\subsection{The Robin Laplacian: complex $\delta$ couplings}
\label{subsec:robin-laplace}

To define our operator, we first need to identify a distinguished set of vertices, which will be equipped with our Robin-type condition: we fix an arbitrary set $\V_R=\{v_1,\dots,v_k\} \subset \V$ with cardinality $k\leq n:=|\V|$ and a vector $\alpha = (\alpha_1,\dots,\alpha_k) \in \C^k$ with $\alpha_j=\alpha(v_j)$, $j=1,\dots,k$, and define a sesquilinear form $a_\alpha : H^1 (\G) \times H^1 (\G) \to \C$ by
\begin{equation}
\label{eq:qg-form}
	a_\alpha [f,g] := \int_{\G} f' \cdot \overline{g}'\,\dx + \sum_{j=1}^k \alpha_j f(v_j) \overline{g(v_j)}, \qquad f,g\in H^1(\G),
\end{equation}
with the integral defined as in \eqref{eq:integral-graph}. A simple integration by parts shows that the operator on $L^2(\G)$ associated with this form is the Laplacian, i.e., $-\frac{\mathrm{d}^2}{\dx^2}$ on each edge, whose domain consists of those functions $f \in H^1 (\G)$ such that 
\begin{enumerate}
\item $f|_e \in H^2 (e) \hookrightarrow C^1(e)$ for all $e \in \mathcal{E}$, 
\item $f$ is continuous at every vertex $v\in\V$,
\item $f$ satisfies the following vertex conditions:
\begin{enumerate}
\item[(a)] if $v_j \in \V_R$, then
\begin{equation}
\label{eq:qg-robin-condition}
	\sum_{e \sim v_j} \frac{\partial}{\partial\nu} f|_e (v_j) + \alpha_j f(v_j) = 0;
\end{equation}
\item[(b)] if $v_j \in \V \setminus \V_R =: \V_N = \{v_{k+1},\dots,v_n\}$, then
\begin{equation}
\label{eq:qg-kirchhoff-condition}
	\sum_{e \sim v_j} \frac{\partial}{\partial\nu} f|_e (v_j) = 0.
\end{equation}
\end{enumerate}
\end{enumerate}

By way of analogy with its counterparts on domains and manifolds, we will call the unbounded operator on $L^2(\G)$ associated with the form $a_\alpha$ the \emph{Robin Laplacian} (associated with the vertex set $\V_R$ and the coefficient $\alpha$), denoted by $-\Delta_{\V_R}^\alpha$, although as mentioned in the introduction this Robin condition is most commonly known as a $\delta$ condition in the literature. We will also call $\V_R$ the \emph{set of Robin vertices}, consistent with the nomenclature in \cite[Section~1.4.1]{BerkolaikoKuchment}; and we note that condition \eqref{eq:qg-kirchhoff-condition} on non-Robin vertices is the condition usually known as Kirchhoff, which corresponds to the Robin condition \eqref{eq:qg-robin-condition} with $\alpha = 0$. The Kirchhoff condition together with continuity is then known variously in the literature as natural, standard, or even sometimes just Neumann or Neumann--Kirchhoff; it is for this reason that we will use the letter ``$N$'' as an index for the corresponding vertex set $\V_N$. 

Finally, we will say that there is a Dirichlet condition at a vertex $v_j\in \V$ if all functions in the domain of the form or operator are simply equal to zero at $v_j$; no further conditions on the functions are imposed at $v_j$. We will denote by $-\Delta_{\V_0}^D$ the Laplacian satisfying Dirichlet conditions at every vertex in $\V_0\subset\V$ and continuity plus Kirchhoff conditions at all vertices of $\V_N=\V\setminus\V_0$. At the level of sesquilinear forms, the form associated with this operator is given by $a_0$, and its form domain is
\begin{displaymath}
	H^1_0 (\G,\V_0) := \{f \in H^1 (\G): f(v_j) = 0 \text{ for all } v_j \in \V_0 \}.
\end{displaymath}
We will correspondingly write $-\Delta_{\V}^D$ for the Laplacian on $L^2(\G)$ satisfying Dirichlet conditions at every vertex of $\V$, in which case $\G$ decouples to a disjoint union of $m=|\E|$ intervals, each equipped with Dirichlet conditions at the endpoints. We refer in particular to \cite[Section~1.4]{BerkolaikoKuchment} for more details on these operators and vertex conditions.

All the operators $-\Delta_{\V_R}^\alpha$, $-\Delta_{\V_R}^D$ are seen to have compact resolvent (since the embedding of $H^1 (\G)$ into $L^2 (\G)$ is compact), and hence discrete spectrum, for any $\alpha \in \mathbb{C}^k$. For real $\alpha$ or Dirichlet conditions, this is contained in \cite[Theorem~3.1.1]{BerkolaikoKuchment}. For complex $\alpha$, this may be deduced from \cite[Section~3]{BK12} or \cite[Sections~3.5 and~5]{HKS}, or proved directly using the compactness of the embedding $H^1(\G) \hookrightarrow L^2(\G)$ and the fact that $-\Delta_{\V_R}^\alpha$ must have non-empty resolvent set, e.g., by Theorem~\ref{thm:numerical-range}.

As is standard, given any such operator $\A \in \{ -\Delta_{\V_R}^\alpha, -\Delta_{\V_R}^D\}$ we will write $\sigma (\A)$ for its spectrum and $\rho (\A)$ for its resolvent set. For each eigenvalue $\lambda \in \sigma (-\Delta_{\V_R}^\alpha)$, there exists an eigenfunction $\psi \in H^1(\G)$ which satisfies
\begin{equation}
\label{eq:ev-weak-form}
	\int_\G \psi' \cdot \overline{\varphi}'\,\dx + \sum_{j=1}^k \alpha_j \psi(v_j)\overline{\varphi(v_j)} = \lambda \int_\G \psi\overline{\varphi}\,\dx \qquad \text{for all } \varphi \in H^1(\G).
\end{equation}

Throughout, we will assume the connected, compact graph $\G(\V,\E)$ and set $\V_R\subset\V$ of Robin vertices to be fixed. Before continuing, we first note the following basic property of the dependence of the eigenvalues of $-\Delta_{\V_R}^\alpha$ on $\alpha \in \mathbb{C}^k$. 

\begin{lemma}
\label{lem:analytic-eigenvalue-dependence}
The operator family $\A(\alpha)=-\Delta_{\V_R}^\alpha$, $\alpha\in\C$, is self-adjoint holomorphic; in particular, $\A(\alpha)^\ast=\A(\overline{\alpha})$ for all $\alpha \in \C$, and up to possible crossing points, each eigencurve $\lambda(\alpha)$ depends holomorphically on $\alpha$.
\end{lemma}

For the proof, see \cite[Theorems~VII.4.2 and~VII.1.8, Remark~VII.4.7]{Kato}. Analyticity results can also be found in \cite[Section~3.4]{BK12}.

\section{The Dirichlet-to-Neumann operator}
\label{sec:dno}

We now turn to the proof of Theorem~\ref{thm:qg-behaviour}. It is based on the Dirichlet-to-Neumann operator $M(\lambda)$, cf. \cite[~Sections~2 and 7]{BKL}: given a vector (Dirichlet data) $g \in \C^k \sim \V_R$ and $\lambda \not\in \sigma (-\Delta_{\V_R}^D)$, there exists a unique weak solution $f \in H^1 (\G)$ of the Dirichlet problem
\begin{equation}
\label{eq:dirichlet-data-on-vr}
\begin{aligned}
	-f'' &= \lambda f \quad &&\text{edgewise,}\\
	f|_{\V_R} &= g \quad && \\
	\sum_{e\sim v_j} \frac{\partial}{\partial\nu} f|_e (v_j) &= 0 \qquad &&\text{at all } n-k\text{ vertices } v_j\in \V_N.
\end{aligned}
\end{equation}
The Dirichlet-to-Neumann operator $M(\lambda)$ maps given Dirichlet data $g=f|_{\V_R}$ to the corresponding Neumann data $-\frac{\partial}{\partial\nu} f|_e (v_j)$ of the same solution $f$ of the problem \eqref{eq:dirichlet-data-on-vr}, that is, a map from $\V_R$ to itself. If we fix the order $v_1,\ldots,v_k$ of the vertices in $\V_R$, then $M(\lambda)$ is canonically identifiable with a matrix in $\C^{k\times k}$ (and in future we shall make this identification without further comment). We now wish to analyse this operator in more detail.

We first note that we may assume without loss of generality that $\G$ does not have any loops nor multiple parallel edges (i.e., between any two distinct vertices there is at most one edge); indeed, if this is not the case, then we may insert a new, artificial vertex of degree two in the middle of each affected edge. When these vertices are equipped with continuity and Kirchhoff conditions, the Laplacian on the resulting graph is unitarily equivalent to the one on the unaltered graph (see \cite[Section~3]{BKKM17}), and so the Dirichlet-to-Neumann operator on the unaffected set $\V_R$ of Robin vertices is equally unaffected.

Now let $v_i, v_j \in \V$ be any two distinct vertices and suppose they are joined by a (unique) edge $e_{ij}$ having length $\ell_{ij}>0$. It is known, and a short calculation shows, that the Dirichlet-to-Neumann operator associated with the graph consisting just of this edge (that is, an interval of length $\ell_{ij}$) and the parameter $\lambda \in \C \setminus \{ \pi^2 n^2/\ell_{ij}^2: n \in \N \}$ may be represented by the matrix
\begin{equation}
	\label{eq:dno-one-edge}
M_{e_{ij}}(\lambda)=
	\sqrt{\lambda}\begin{pmatrix}
		-\cot \sqrt\lambda \ell_{ij} & \csc \sqrt\lambda \ell_{ij} \\
		\csc \sqrt\lambda \ell_{ij} & -\cot \sqrt\lambda \ell_{ij}
	\end{pmatrix}.
\end{equation}
Fix $\lambda \in \C$, to be specified precisely later. We denote by $\widetilde{M}_{e_{ij}} \in \C^{n\times n}$ the matrix corresponding to the operator \eqref{eq:dno-one-edge} extended by zero to the $n-k$ other vertices in $\V_N=\V\setminus\V_R$. That is, for fixed $1\leq i,j\leq n$, the $(i,i)$- and $(j,j)$-entries of $\widetilde{M}_{e_{ij}}$ are given by
\begin{equation}
\label{eq:aij}
	\sqrt{\lambda} A_{ij}:= -\sqrt{\lambda} \cot \sqrt\lambda \ell_{ij} ;
\end{equation}
the $(i,j)$- and $(j,i)$-entries of $\widetilde{M}_{e_{ij}}$ are given by
\begin{equation}
\label{eq:bij}
	\sqrt{\lambda}B_{ij}:= \sqrt{\lambda}\csc \sqrt\lambda \ell_{ij},
\end{equation}
and all other entries are zero; that is,

\begin{equation*}
\widetilde{M}_{e_{ij}}(\lambda) = \sqrt{\lambda}
\begin{pmatrix}
0 & \cdots & \cdots & \cdots & 0  \\
\vdots & A_{ij} & 0 & B_{ij} & \vdots  \\
\vdots & 0 & \cdots & 0 & \vdots \\
\vdots & B_{ij} & 0 & A_{ij} & \vdots \\
0 & \cdots & \cdots & \cdots & 0  
\end{pmatrix} \in \C^{n\times n}.
\end{equation*}

We may then represent the Dirichlet-to-Neumann operator $M_{\V} (\lambda)$ acting on \emph{all} vertices of $\G$, that is, $M_{\V} (\lambda) \in \C^{n\times n}$, by summing over all these matrices,
\begin{equation}
\label{eq:dno-all-vertices}
	M_{\V} (\lambda) = \sum_{e\in\E} \widetilde{M}_{e}(\lambda),
\end{equation}
which is well defined as long as $\lambda$ is not in $\sigma(-\Delta_{\V}^D)$, i.e., not in the Dirichlet spectrum of any of the decoupled edges considered as a collection of disjoint intervals. If we set $A_{ij} = B_{ij} = 0$ whenever there is no edge joining $v_i$ and $v_j$, then we may explicitly write the $(i,j)$-entry of $M_{\V} (\lambda)$ as
\begin{equation}
\label{eq:dno-all-vertices-ij-entry}
	\left(M_{\V} (\lambda)\right)_{ij} = \sqrt{\lambda} \begin{cases} \sum_{p=1}^n A_{ip} \qquad &\text{if } i=j,\\
	B_{ij} \qquad &\text{if } i \neq j. \end{cases}
\end{equation}
It is immediate that $M_{\V} (\lambda)$ depends analytically on $\lambda$, with isolated singularities at the discrete set $\sigma(-\Delta_{\V}^D)$.

Importantly, the Dirichlet-to-Neumann matrix $M(\lambda)$ acting on $\V_R$ can be written in a natural way in terms of $M_{\V} (\lambda)$. We recall that $\V$ consists of the (ordered) vertices $v_1,\ldots,v_n$, such that the first $k$ entries $\V_R = \{v_1,\ldots,v_k\}$ are equipped with the Robin boundary condition. With this ordering, we write $M_{\V} (\lambda)$ in block form as
\begin{equation}
\label{eq:dno-block-form}
	M_{\V} (\lambda) = \begin{pmatrix} R & C^T \\ C & K \end{pmatrix},
\end{equation}
where $R \in \C^{k\times k}$ represents the restriction of $M_{\V}$ to the $k$ Robin vertices, $K \in \C^{(n-k) \times (n-k)}$ is the restriction to the remaining $n-k$ (``Kirchhoff'') vertices, and $C \in \C^{(n-k) \times k}$ and its transpose $C^T$ give the interaction (``coupling'') between these two groups of vertices. The following representation is adapted from \cite{DGK}, although we expect it is well known elsewhere.

\begin{lemma}
\label{lem:dno-qg-representation}
With the representation \eqref{eq:dno-block-form}, the matrix $K$ is invertible if and only if $\lambda \not\in \sigma (-\Delta_{\V_R}^D)$. Whenever $K$ is invertible, the operator $M (\lambda)$ is well defined and may be represented in matrix form by
\begin{equation} 
\label{eq:dno-qg-representation}
	M (\lambda) = R - C^T K^{-1} C.
\end{equation}
\end{lemma}

\begin{proof}
Let $x^R\in\C^k\sim\V_R$ be the Dirichlet data $g$ from \eqref{eq:dirichlet-data-on-vr}. If we write $x=(x^R,x^N)^T\in\C^n$ and
\begin{displaymath}
f|_{\V_N} =: (x_{k+1},\dots,x_n)^T=:x^{N} \in \C^{n-k} \simeq \V_N,
\end{displaymath}
which is well defined and thus uniquely determined by $x^{R}=(x_1,\dots,x_k)^T\in\C^k$ since $\lambda \not\in \sigma (-\Delta_{\V_R}^D)$, then by construction
\begin{equation}
\label{eq:blockmatrix-rep-dno-proof}
	\begin{pmatrix} R & C^T \\ C & K \end{pmatrix}\begin{pmatrix} x^{R} \\ x^{N} \end{pmatrix} =
	\begin{pmatrix} \sum_{e\sim v} \frac{\partial}{\partial\nu} f|_e (v) \\ 0 \end{pmatrix} =
	\begin{pmatrix} M(\lambda) x^{N} \\ 0 \end{pmatrix} \in\C^n.
\end{equation}
That is, $M(\lambda)x^{N} = Rx^{R} + C^Tx^{N}$, where $Cx^{R} + Kx^{N} = 0$. Since $x^{N}$ is uniquely determined by $x^{R}\in \C^k$ arbitrary, we must have that $K$ is invertible, $x^{N} = -K^{-1}Cx^{R}$, and thus \eqref{eq:dno-qg-representation} follows, if $\lambda \not\in \sigma (-\Delta_{\V_R}^D)$. If on the other hand $\lambda \in \sigma (-\Delta_{\V_R}^D)$, then since $x^{N}$ is no longer uniquely determined by $x^{R}$ in general (if $\psi$ is an eigenfunction of $-\Delta_{\V_R}^D$, then $x^{N}+\psi|_{\V_N}$ is also a solution), $K$ cannot be invertible.
\end{proof}
We can now state the central duality result linking the eigenvalues of $M(\lambda)$ and $-\Delta_{\V_R}^\alpha$. Here we will suppose that the vector $\alpha = (\alpha_1,\ldots,\alpha_k) \in \C^k$ is given and assume that the vertex $v_j \in \V_R = \{v_1,\ldots,v_k\}$ is equipped with the Robin parameter $\alpha_j$; for brevity we will then write
\begin{displaymath}
	I_\alpha := \diag \{ \alpha_1, \ldots, \alpha_k \} \in \C^{k \times k}.
\end{displaymath}
The next statement is well known in the case of real $\alpha \in \R$ (see \cite[Theorem~3.5.2]{BerkolaikoKuchment}); the proof in the complex vector case $\alpha \in \C^k$ is identical, and we omit it.

\begin{theorem}
	\label{thm:bk-duality}
Let $\lambda\in\rho(-\Delta_{\V_R}^D)$. Then $\lambda\in\sigma(-\Delta_{\V_R}^\alpha)$ if and only if
\begin{displaymath}
\det(M(\lambda)-I_\alpha)=0.
\end{displaymath}
\end{theorem}

\section{Asymptotics of the Dirichlet-to-Neumann operator}
\label{sec:dno-asymptotics}

We now investigate what happens to $M(\lambda)$ when $\lambda\to\infty$. We first note the following trivial but useful implication of Lemma~\ref{lem:dno-qg-representation}.

\begin{lemma}
\label{lem:dno-qg-poles}
The Dirichlet-to-Neumann matrix $M(\lambda)$ is a meromorphic function of $\lambda$. It is well defined for all $\lambda \in \rho (-\Delta_{\V_R}^D)$, and each $\lambda \in \sigma (-\Delta_{\V_R}^D)$ is a pole of finite order of $M(\lambda)$.
\end{lemma}

For a vector $z=(z_1,\dots,z_k)\in\C^k$ we denote by $\GM(z)$ the smallest of the moduli of its components $z_j$, $j=1,\dots,k$, i.e.
\begin{displaymath}
\GM(z) = \min_{j=1,\dots,k} |z_j|.
\end{displaymath}
Using the duality between the eigenvalues $\lambda \in \sigma (-\Delta_{\V_R}^\alpha)$ of the Robin Laplacian and the eigenvalues $\alpha \in \sigma (M(\lambda))$ of the Dirichlet-to-Neumann matrix from Theorem~\ref{thm:bk-duality}, we obtain

\begin{theorem}
	\label{thm:dno-qg-asymptotics-duality}
For any compact graph $\G$ and any bounded set $\Omega\subset\C$ such that
\begin{displaymath}
\dist(\Omega,\sigma(-\Delta_{\V_R}^D))>0
\end{displaymath}
there exists a number $\hat\alpha>0$ depending only on $\Omega$, $\G$, and $\V_R$ such that
\begin{displaymath}
\sigma(-\Delta_{\V_R}^\alpha)\cap \Omega=\emptyset
\end{displaymath}
for all $\alpha = (\alpha_1,\ldots,\alpha_k) \in \C^k$ such that $\GM(\alpha)>\hat\alpha$.
\end{theorem}
This theorem immediately implies the following dichotomy.
\begin{corollary}
	\label{cor:dno-qg-asymptotics-duality}
Suppose $\GM(\alpha) \to \infty$ and $\lambda = \lambda (\alpha)$ is an analytic branch of eigenvalues of $-\Delta_{\V_R}^\alpha$. Then either $\lambda \to \infty$ in $\C$ or $\lambda$ converges to a point in $\sigma (-\Delta_{\V_R}^D)$ as $\GM(\alpha) \to \infty$.
\end{corollary}

\begin{proof}[Proof of Theorem~\ref{thm:dno-qg-asymptotics-duality} and hence of Corollary~\ref{cor:dno-qg-asymptotics-duality}]
For $\GM(\alpha)>0$ we consider the invertibility of 
\begin{equation}\label{eq:I_alpha_new}
M(\lambda)-I_\alpha = I_\alpha\left(I_{\alpha^{-1}} M(\lambda)- I\right),
\end{equation}
where $I_{\alpha^{-1}} := \diag \{ \alpha_1^{-1}, \ldots, \alpha_k^{-1} \}$ is well defined (since $\GM(\alpha)>0$) and satisfies $I_{\alpha^{-1}}=I_\alpha^{-1}$ by definition. Now the matrix $M(\lambda) \in \C^{k \times k}$ is a meromorphic function of $\lambda$ with singularities at $\sigma (-\Delta_{\V_R}^D)$ (see Lemma~\ref{lem:dno-qg-poles}), and hence its norm is uniformly bounded on $\Omega \subset\!\subset \rho(-\Delta_{\V_R}^D)$. Hence for each such $\Omega$ there exists a constant $c_\Omega>0$ independent of $\alpha\in\C^k$ such that
\begin{displaymath}
\begin{aligned}
	\sup_{\lambda\in \Omega}\left\| I_{\alpha^{-1}} M(\lambda)\right\|_{\C^k\rightarrow \C^k} 
	&\leq \sup_{\lambda\in \Omega} \left\|I_{\alpha^{-1}}\right\|_{\C^k\rightarrow \C^k}\left\|M(\lambda)\right\|_{\C^k\rightarrow \C^k} \\
	&= c_\Omega\left\|I_{\alpha^{-1}}\right\|_{\C^k\rightarrow \C^k} \;\longrightarrow\; 0
\end{aligned}
\end{displaymath}
as $\GM(\alpha)\rightarrow\infty$. This convergence implies that $\left(I_{\alpha^{-1}} M(\lambda)- I\right)$ is invertible (as a Neumann series) and the right-hand side of \eqref{eq:I_alpha_new} is, too. That is, there exists a constant $\hat\alpha>0$ such that 
\begin{displaymath}
	\GM(\alpha) >\hat\alpha \quad\Longrightarrow\quad M(\lambda)-I_\alpha \text{ is invertible.}
\end{displaymath}
In particular, the kernel of $M(\lambda)-I_\alpha$ is trivial for $\GM(\alpha)>\hat{\alpha}$, and thus there exist no eigenvalues of the Robin Laplacian in $\Omega$ by Theorem~\ref{thm:bk-duality}.
\end{proof}

It remains to analyse what divergent behaviour is possible, and under what circumstances. To this end, we use the representations \eqref{eq:dno-all-vertices-ij-entry}, \eqref{eq:dno-block-form} and \eqref{eq:dno-qg-representation} together with ideas drawn from \cite[Section~2]{BKL} for the interval.

We first note that since the coefficients $A_{ij}$ and $B_{ij}$ given by \eqref{eq:aij} and \eqref{eq:bij}, respectively, are periodic in $\Re \sqrt{\lambda}$, we only need to consider the case $\Im \sqrt{\lambda} \to \pm \infty$, in which case we have the asymptotics 
\begin{equation}
\label{cot-asymptotics}
	\cot z = \I \left(1+ \frac{2}{\e^{2\I z}-1} \right) = \mp \I + \mathcal{O} \left( \e^{\mp 4\Im z}\right)
\end{equation}
and
\begin{equation}
\label{csc-asymptotics}
	\csc z = \frac{2\I}{\e^{\I z} - \e^{-\I z}} = \mathcal{O} \left(\e^{\mp 2 \Im z}\right)
\end{equation}
as $\Im z \to \pm \infty$, independently of $\Re z$. For $z=\frac{\ell_{ij}}{2}\sqrt\lambda$ this gives the following asymptotic expansion of $M(\lambda) \in \C^{k\times k}$. In what follows, for brevity we will set 
\begin{align*}
D&:=\diag\{\deg v_1, \ldots, \deg v_k\}\in \N^{k \times k},\\
\tilde{D}&:=\diag\{\deg v_{k+1}, \ldots, \deg v_n\}\in \N^{(n-k)\times(n-k)}.
\end{align*}

\begin{lemma}
\label{lem:qg-dno-asymptotics}
Suppose $\lambda \to \infty$ in $\C$ in such a way that $\Im \sqrt{\lambda} \to \pm\infty$, and recall the definition $\ell_\G := \min\{\ell_e : e\in\E\} > 0$. Then $M(\lambda)$ has the asymptotic expansion
\begin{equation}
\label{eq:qg-dno-asymptotics}
	M(\lambda) = \pm \I \sqrt{\lambda} D + \mathcal{O}\left(\sqrt{\lambda}\e^{\mp \ell_\G \Im \sqrt{\lambda}}\right).
\end{equation}
\end{lemma}

\begin{proof}[Proof of Lemma~\ref{lem:qg-dno-asymptotics}]
Recall the matrices $R$, $C$ and $K$ introduced in \eqref{eq:dno-block-form}. Then the expression \eqref{eq:dno-all-vertices-ij-entry} for the coefficients of these matrices plus the asymptotics
\begin{displaymath}
	A_{ij} = \pm \I + \mathcal{O}\left(\e^{\mp 2\ell_{ij} \Im \sqrt{\lambda}}\right), \qquad
	B_{ij} = \mathcal{O}\left(\e^{\mp \ell_{ij} \Im \sqrt{\lambda}}\right)
\end{displaymath}
as $\Im \sqrt{\lambda} \to \pm \infty$, respectively, which follow from \eqref{cot-asymptotics} and \eqref{csc-asymptotics}, imply that
\begin{displaymath}
	R = \pm \I \sqrt{\lambda} D + \mathcal{O} \left(\sqrt{\lambda} \e^{\mp \ell_\G \Im \sqrt{\lambda}}\right),
\end{displaymath}
as well as
\begin{displaymath}
	C,\, C^T = \mathcal{O} \left(\sqrt{\lambda} \e^{\mp \ell_\G \Im \sqrt{\lambda}}\right)
\end{displaymath}
and
\begin{displaymath}
	K = \pm \I \sqrt{\lambda} \tilde{D} + \mathcal{O} \left(\sqrt{\lambda} \e^{\mp \ell_\G \Im \sqrt{\lambda}}\right).
\end{displaymath}
From the latter, we obtain via an easy argument that $K^{-1} = \mathcal{O}(1/\sqrt{\lambda})$ and hence also
\begin{equation}
\label{eq:CTK-1C}
	C^TK^{-1}C = \mathcal{O}\left(\sqrt{\lambda} \e^{\mp 2\ell_\G \Im \sqrt{\lambda}}\right).
\end{equation}
Combined with the asymptotic expansion for $R$ and the representation \eqref{eq:dno-qg-representation} of $M(\lambda)$, this immediately yields \eqref{eq:qg-dno-asymptotics}.
\end{proof}

As a corollary of Lemma~\ref{lem:qg-dno-asymptotics} we obtain that the $k$ eigenvalues $\alpha_1, \ldots, \alpha_k$ of $M(\lambda)$ satisfy
\begin{equation}
\label{eq:qg-dno-eigenvalue-asymptotics}
	\alpha_j = \pm \I \sqrt{\lambda} \deg v_j + \mathcal{O}\left(\sqrt{\lambda}\e^{\mp \ell_\G \Im \sqrt{\lambda}}\right)
\end{equation}
as $\Im \sqrt\lambda \to \pm \infty$, $j=1,\ldots,k$; in fact, convergence of the corresponding eigenvectors of $M(\lambda)$ to those of $D$, that is, to the standard basis of $\C^k$, also follows, but we will not need this. In other words, if the spectral parameter $\lambda \to \infty$ in such a way that its distance to the positive real semi-axis tends to $\infty$ (corresponding to $\Im \sqrt\lambda \to \pm \infty$), then, counting possible multiplicities, we obtain $k$ curves $\alpha_j = \alpha_j (\lambda)$, $j=1,\ldots,k$, each described asymptotically by the formula \eqref{eq:qg-dno-eigenvalue-asymptotics}. To prove Theorem~\ref{thm:qg-behaviour}, it remains to ``invert'' these asymptotics, that is, express these curves as functions of $\alpha_j$. For this part of the argument, we may essentially appeal to the proof given in \cite[Section~9.1.3]{BKL} for the corresponding statement on the interval.

\begin{proof}[Proof of Theorem~\ref{thm:qg-behaviour}]
We assume that $\alpha=(\alpha_1,\dots,\alpha_k) \to \infty$ in $\C^k$ and recall the two cases 
\begin{enumerate}
\item $\alpha_j \to\infty$ in a sector fully contained in the open left half-plane, for all $1\leq j\leq m$;
\item $\alpha_j \to\infty$ such that $\Re\alpha_j$ remains bounded from below as $\alpha_j\to\infty$, for all $m+1\leq j\leq k$.
\end{enumerate}
We wish to show that for each $\alpha_1,\dots,\alpha_m\to\infty$ there exists a corresponding eigenvalue $\lambda_j=\lambda(\alpha_j)$ (here and throughout the proof we understand ``eigenvalue'' to mean ``analytic curve of eigenvalues'') which behaves as asserted and that these $m$ distinct eigenvalues $\lambda_1,\dots,\lambda_m$ are the only ones which diverge away from the positive real semi-axis.

Suppose first that $\lambda$ is such an eigenvalue diverging away from the positive real semi-axis; then necessarily $\Im\sqrt{\lambda}\to\pm\infty$. By \eqref{eq:qg-dno-eigenvalue-asymptotics} we obtain $k$ eigenvalues of $M(\lambda)$, $\alpha_1,\ldots,\alpha_k$, behaving like $\alpha_j \sim \pm\I\sqrt{\lambda}\deg v_j$.

Now fix $j=1,\dots,m$. By the same inversion argument based on Rouch\'e's theorem that was used in \cite[Section~9.1.3]{BKL}, there exists an eigenvalue $\lambda=\lambda(\alpha_j)$ of $-\Delta_{\V_R}^\alpha$ satisfying $\lambda\sim -\alpha_j^2/(\deg v_j)^2$ and which has the asymptotical error term $\mathcal{O}\left(\alpha_j^2 \e^{\ell_G\Re\alpha_j}\right)$ as $\alpha_j\to\infty$. Since this works for each $\alpha_j$ which diverges as described in (1), we arrive at $m$ divergent eigenvalues, each of which satisfies \eqref{eq:qg-divergent-ev}.

Suppose now that there is an additional, $(m+1)$st divergent eigenvalue $\lambda = \lambda (\alpha)$ which satisfies $\Im \sqrt\lambda \to \pm \infty$. Then, again, the matrix $M(\lambda)$ has $k$ eigenvalues satisfying \eqref{eq:qg-dno-eigenvalue-asymptotics}. By assumption, $\lambda$ is not an eigenvalue of $-\Delta_{\V_R}^\alpha$ corresponding to the $m$ curves found above, that is, it does not correspond to $\alpha_1,\ldots,\alpha_m$. Hence, applying the same inversion procedure, there must be some $j_0 \in \{m+1,\ldots,k\}$ such that $\lambda$ corresponds to the eigenvalue $\alpha_{j_0} \leftrightarrow \lambda$ described asymptotically by \eqref{eq:qg-dno-eigenvalue-asymptotics}. But now a short argument shows that the condition $\Im \sqrt\lambda \to \pm \infty$ together with the relation \eqref{eq:qg-dno-eigenvalue-asymptotics} implies that necessarily $\Re\alpha_{j_0} \to -\infty$ as $\lambda \to \infty$. This contradicts the assumption (2), and we conclude that no such divergent eigenvalue $\lambda$ can exist which is not already among the $m$ found above.

Finally, we already know from Corollary~\ref{cor:dno-qg-asymptotics-duality} that each eigenvalue of $-\Delta_{\V_R}^\alpha$ which does not diverge to $\infty$ converges to some eigenvalue of the Dirichlet Laplacian $-\Delta_{\V_R}^D$ as $\alpha \to \infty$.
\end{proof}

\begin{remark}
\label{rem:kirchhoff-in-VN}
Let us finish by discussing the role of the standard (continuity-Kirchhoff) conditions that were assumed to hold on the non-Robin vertex set $\V_N = \V \setminus \V_R$. There are two key places where these conditions enter: the block matrix representation of Lemma~\ref{lem:dno-qg-representation}, and the subsequent asymptotics of the Dirichlet-to-Neumann operator (Lemma~\ref{lem:qg-dno-asymptotics}, in particular \eqref{eq:CTK-1C}).

Suppose that the functions in the domain of $-\Delta_{\V_R}^\alpha$ should satisfy other (local) vertex conditions at $v\in\V_N$. This would, in general, result in a different matrix $\widetilde{K}$ in the representations \eqref{eq:dno-block-form} and \eqref{eq:blockmatrix-rep-dno-proof}, but also, more importantly, the last $n-k$ components $Cx^R+\widetilde{K} x^N$ of \eqref{eq:blockmatrix-rep-dno-proof} may no longer vanish, meaning that $x^N$ may not be easily expressible as a function of $x^R$: any vertex condition such that $Cx^R+\widetilde{K}x^N$ depends on $f$ (e.g. any Robin condition) may result in $x^N$ no longer being uniquely determined by $x^R$ and $\widetilde K$ not being invertible. Thus for the proof used here to work, we require the vertex condition to satisfy $Cx^R+\widetilde{K}x^N=0$.

In addition, to obtain the correct asymptotic behaviour \eqref{eq:qg-dno-asymptotics} of the Dirichlet-to-Neumann matrix, we require that $C^T\widetilde{K}^{-1}C$ not influence the leading order of the asymptotics of $M(\lambda)$, which in turn requires that $\widetilde K$ not decrease too rapidly as $\Im\sqrt{\lambda}\to\pm\infty$.

We will not explore further the question of what other vertex conditions might satisfy these two conditions.
\end{remark}

\section{Estimates on the numerical range and the eigenvalues}
\label{sec:numerical-range}

Here we wish to complement the asymptotic behaviour of the divergent eigenvalues described by Theorem~\ref{thm:qg-behaviour} with concrete estimates on the location of the eigenvalues. We will present three sets of results which, while perhaps not surprising, give a fairly complete picture of the location spectrum for every fixed $\alpha$. We first consider the location of the so-called numerical range; we recall that for $\V_R=\{v_1,\dots,v_k\}\subset\V$ and the corresponding parameter vector $\alpha=(\alpha_1,\dots,\alpha_k)\in\C^k$ (with $\alpha_j=\alpha(v_j)$) the \emph{numerical range} of the form $a_\alpha$ given by \eqref{eq:qg-form} is, by definition, the set
\begin{displaymath}
	W(a_\alpha) = \{ a_\alpha [f,f]: \|f\|_2 = 1\} 
	= \left\{ \int_\G |f'|^2\,\dx + \sum_{j=1}^k \alpha_j |f(v_j)|^2 : \, \int_\G |f|^2\,\dx = 1\right\} \subset \C,
\end{displaymath}
and that clearly every eigenvalue of the operator $-\Delta_{\V_R}^\alpha$ is in $W(a_\alpha)$. Our first results give an estimate on the location of the set $W(a_\alpha)$ in the complex plane, analogous to those of \cite[Section~6]{BKL} for the complex Robin Laplacian on a domain in $\R^d$. This leads to bounds on the real part of the eigenvalues which are, in particular, sharp up to the first term of the asymptotics as $\alpha \to \infty$ in $\C^k$. In addition to these bounds, we also consider more precise estimates on the imaginary part of the eigenvalues afterwards.

For the numerical range, we consider the case of $\alpha \in \C^k$ and the case of vertex-independent $\alpha_1=\ldots=\alpha_k=:\alpha \in \C$ separately. Notationally, for the fixed set $\V_R=\{v_1,\dots,v_k\}$ of Robin vertices we will always write
\begin{displaymath}
	\MD := \min_{j=1,\dots,k}\deg v_j .
\end{displaymath}
We also recall that $\ell_\G = \min \{\ell_e : e \in \E\}$ is the length of the shortest edge in $\G$. The proofs of all the following statements will be deferred to Section~\ref{sec:numerical-range-proofs}.

\begin{theorem}[The numerical range]
\label{thm:numerical-range}
\begin{enumerate}
\item Let $\alpha \in \C^k$. Then the numerical range $W(a_\alpha)$, and in particular every eigenvalue of $-\Delta_{\V_R}^\alpha$, is contained in the set
\begin{equation}
\label{eq:numerical-range-variable-alpha}
	\Lambda_{\G,\alpha} := \left\{ t+\sum_{j=1}^k\alpha_j s_{j} \in \C : t\geq 0, \;s_{j}\in \left[0,\frac{2}{\MD}\sqrt{\tau_j} + \frac{2}{\MD\ell_\G}\right]\right\},
\end{equation}
where the numbers $0\leq \tau_j \leq t$ satisfy $\sum_{j=1}^k \tau_j\leq t$.
\item If $\alpha_1=\ldots=\alpha_k=:\alpha \in \C$ is independent of $j=1,\ldots,k$, then $W(a_\alpha)$ is contained in
\begin{equation}
\label{eq:numerical-range-constant-alpha}
	\Lambda_{\G,\alpha} := \left\{ t+\alpha\cdot s \in \C : t\geq 0, s\in \left[0,\frac{2}{\MD}\sqrt{t} + \frac{1}{\MD \ell_\G}\right]\right\}.
\end{equation}
\end{enumerate}
\end{theorem}

Note that if $\Re\alpha_j= \Re \alpha (v_j) \geq 0$ for all $v_j \in \V_R$, then $\Re \lambda \geq 0$ automatically as well, whereas if the components of $\Re \alpha$ are all negative or of indefinite sign, then $\Re\lambda$ may be negative. The set $\Lambda_{\G,\alpha}$ is depicted in Figure~\ref{fig:SimpleCurve} in the simple case that $\alpha \in \C$ with $\Re\alpha, \Im\alpha > 0$.

\begin{figure}[h]
  \includegraphics[scale=0.25]{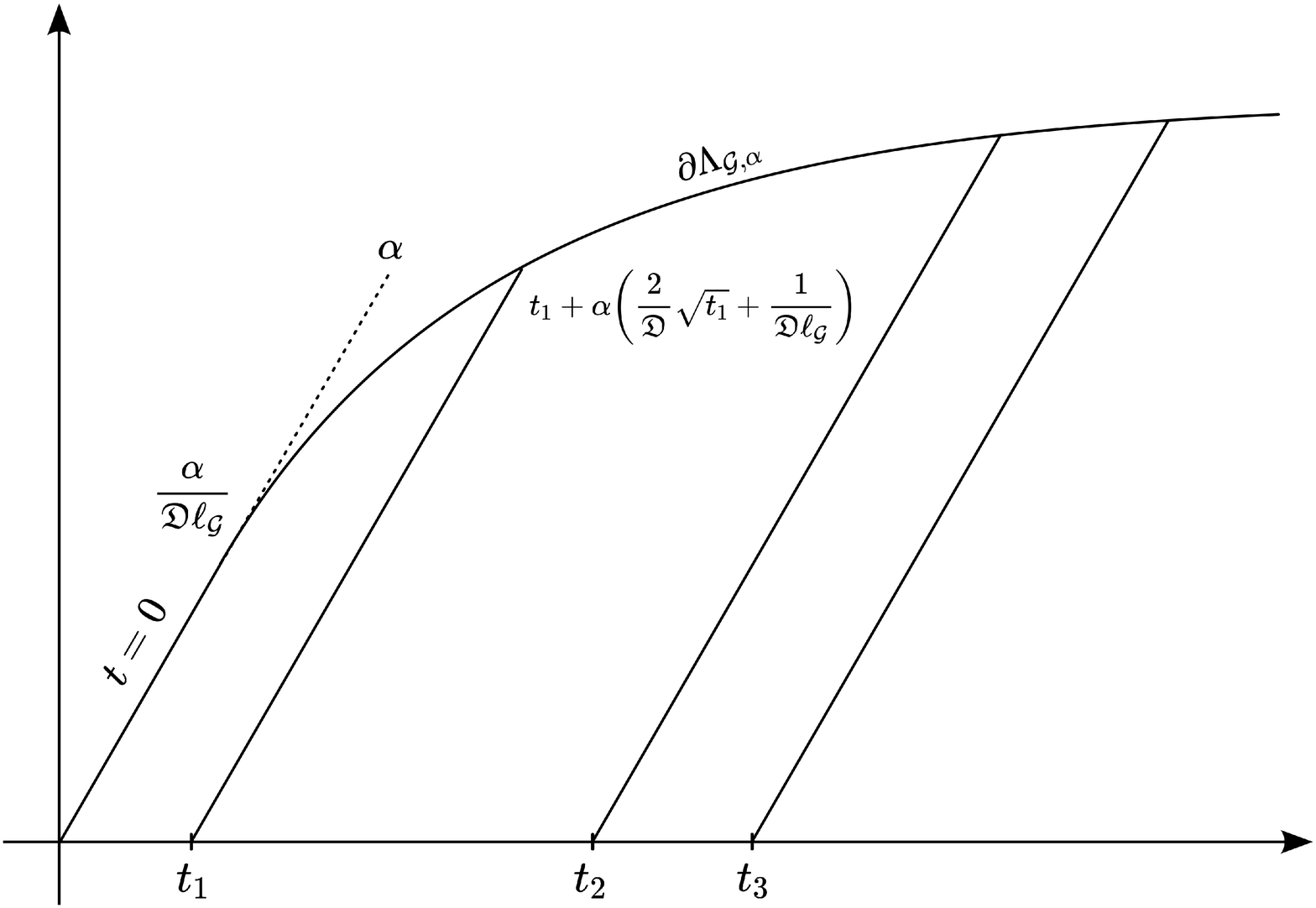}
  \caption{The set $\Lambda_{\G,\alpha}$ from Theorem~\ref{thm:numerical-range}(2), which contains the numerical range $W(a_\alpha)$, for a representative choice of $\Re \alpha >0$ and $\Im \alpha >0$, corresponding to the region between the curve $\partial\Lambda_{\G,\alpha}$ and the real axis. The region is composed of the union of segments of the form $\{ t+\alpha\cdot s \in \C: s\in [0,2\sqrt{t}/\MD+1/\MD\ell_G] \}$, each of slope $\Im\alpha/\Re\alpha$, for different values of $t \geq 0$; the parallel lines show these segments for selected values of $t_1,t_2, t_3 > 0$. Their endpoints form a parabolic section of $\partial\Lambda_{\G,\alpha}$ open to the right.}
  \label{fig:SimpleCurve}
\end{figure}

We now turn to the estimates on the real part of the eigenvalues announced above, which also demonstrate the asymptotic optimality of the bounds on $\Lambda_{\G,\alpha}$ (see Remark~\ref{rem:est-real}). For simplicity, in what follows, we will assume that $\alpha_1 = \ldots = \alpha_k =: \alpha \in \C$ is independent of $j=1,\ldots,k$; a similar statement holds in the general case.

\begin{corollary}[The real part of the eigenvalues]
\label{cor:est-real}
Let $\alpha \in \C$ such that $\Re \alpha < 0$. Then any eigenvalue $\lambda \in \sigma(-\Delta_{\mathcal{V}_R}^\alpha)$ satisfies
\begin{equation}
\label{eq:est-real}
	\Re \lambda \geq - \frac{(\Re\alpha)^2}{\MD^2} + \frac{\Re \alpha}{\MD\ell_\G}.
\end{equation}
\end{corollary}

\begin{remark}
\label{rem:est-real}
Theorem~\ref{thm:qg-behaviour} implies the existence of an eigenvalue behaving like $-\alpha^2/{\MD^2}$ as $\Re\alpha \to -\infty$, meaning that the first term in \eqref{eq:est-real} is correct in this regime. Actually, in the case of real negative $\alpha$ a test function argument can be used to give a complementary upper bound on the smallest (real) eigenvalue $\lambda_1 (\alpha) := \min \sigma(-\Delta_{\V_R}^\alpha)$; namely, we have
\begin{equation}
\label{eq:est-real-upper}
	\lambda_1 (\alpha) \leq \begin{cases} -\frac{\alpha^2}{\MD^2} - \frac{2\alpha}{\MD \ell_\G} - \frac{1}{\ell_\G^2} = -\left[\frac{\alpha}{\MD} + \frac{1}{\ell_\G}\right]^2
	\qquad &\text{if } \alpha < -\frac{\MD}{\ell_\G} < 0,\\
	\frac{k\alpha}{|\G|} \qquad &\text{for all } \alpha < 0,\end{cases}
\end{equation}
where $k = |\mathcal{V}_R|$ is the number of Robin vertices and $|\G|$ is the total length of $\G$; we will prove \eqref{eq:est-real-upper} in Section~\ref{sec:numerical-range-proofs}. Regarding the second term, we observe that as $\alpha \to -\infty$, we have $\lambda_1 (\alpha) = -\alpha^2/\MD^2 + o (\alpha^{-\infty})$; while as $\alpha \to 0$, since $\lambda_1'(0) = 1/|\G|$ (see \cite[Proposition~3.1.6]{BerkolaikoKuchment} and use that the eigenfunctions for $\lambda_1(0)=0$ are constant),
\begin{displaymath}
	\lambda_1(\alpha) = \frac{k\alpha}{|\G|} + \mathcal{O}(\alpha^2) \qquad \text{as } \alpha \to 0.
\end{displaymath}
Hence there can be no ``correct'' coefficient $c \in \R$ of $\alpha$ in any (upper or lower) bound of the form $-\alpha^2/\MD^2 + c\alpha$ which is valid for all $\alpha<0$ and asymptotically sharp for $\alpha \to 0$ and $\alpha \to -\infty$.
\end{remark}

We finish with a more precise statement about the imaginary parts of the eigenvalues.

\begin{theorem}[The imaginary part of the eigenvalues]
\label{thm:imaginary-control}
Let $\alpha \in \C^k$.
\begin{enumerate}
\item If $\Re\alpha_j \geq 0$ for all $j=1,\dots,k$, then any eigenvalue $\lambda$ of $-\Delta_{\mathcal{V}_R}^\alpha$ satisfies
\begin{equation}
\label{eq:imaginary-control-positive}
	|\Im \lambda| \leq \max_{j=1,\dots,k} \frac{|\Im \alpha_j|}{\deg v_j} \left[2\sqrt{\Re\lambda} + \frac{1}{\MD\ell_\G}\right].
\end{equation}
\item If $\Re \alpha_j<0$ for at least one $j=1,\dots,k$, then for every $0 < \varepsilon < 1$ there exists a constant $C=C(\varepsilon)>0$ depending on $\G$ and each $\Re\alpha_j<0$ such that
\begin{equation}
\label{eq:imaginary-control-negative}
	|\Im \lambda| \leq \max_{j=1,\dots,k} \frac{|\Im \alpha_j|}{\deg v_j} \left[2(1-\varepsilon)\sqrt{\Re\lambda+C} + \frac{1}{\MD\ell_\G}\right].
\end{equation}
\end{enumerate}
\end{theorem}

\section{Proofs of the estimates}
\label{sec:numerical-range-proofs}

In this section we prove the statements collected in Section~\ref{sec:numerical-range}. Theorems~\ref{thm:numerical-range} and~\ref{thm:imaginary-control} will follow from a kind of ``trace-type'' inequality which allows us to control precisely the value that an arbitrary $H^1$-function takes at a given vertex in terms of certain subgraphs around it. This should be compared with \cite[Lemma~6.5]{BKL}. To this end, we first require some notation. Let $\xi:\E\to (0,1]$ be an edge-dependent length scaling factor. Given any vertex $v_j \in \mathcal{V}$, we denote by
\begin{displaymath}
	\Scal_j^\xi := \bigcup_{e \sim v_j} \xi(e)e
\end{displaymath}
the star subgraph of $\G$ whose central vertex is $v_j$ and whose pendant edges are the edges $e$ incident with $v_j$, scaled by the factor $\xi(e) \in (0,1]$. We will always make the identification that $\Scal_j^\xi$ is a subgraph of $\G$; in particular, we will treat the scaled edge $\xi(e)e \subset \Scal_j^\xi$ as a subset of the edge $e \subset \G$. In particular, for $\xi(e)\equiv 1$, the star $\Scal_j^1$ is the union of all edges in $\G$ incident with $v_j$; call this the \emph{spanning star} at $v_j$. For an arbitrary collection $\V_0 \subset \V$ of vertices, we denote by $\G_0$ the subgraph consisting of the union of all spanning stars of all vertices $v \in \V_0$. (For example, if $\G$ is a star, then the spanning star of the central vertex is the whole of $\G$, while the spanning star of any of the degree one vertices is just a single edge.) We also define, for any subgraph $\G'=(\V',\E')$ of $\G$,
\begin{displaymath}
	\ell_{\mathcal{G}'}:=\min\limits_{e\in\E'} \ell_e
\end{displaymath}
as the length of the shortest edge $e$ of $\G'$; and, as usual, we set $\|f\|_{\G'}:=\|f\|_{\G',2}$ to be the $L^2(\G')$-norm of $f$. 
\begin{lemma}
\label{lem:graph-trace-estimate}
Let $\xi:\E\to (0,1]$ and $v_j \in \mathcal{V}$ be arbitrary and denote by $\Scal_j^\xi$ the scaled star at $v_j$ as described above. Then
\begin{equation}
\label{eq:graph-trace-local}
	\deg v_j |f(v_j)|^2 \leq 2 \|f\|_{\Scal_j^\xi}\|f'\|_{\Scal_j^\xi} + \frac{1}{\ell_{\Scal_j^\xi}} \|f\|_{\Scal_j^\xi}^2.
\end{equation}
for all $f \in H^1(\G)$. Moreover, if $\mathcal{V}_0 \subset \mathcal{V}$ is an arbitrary set of vertices of $\G$ and $\G_0$ is the subgraph union of spanning stars for $\V_0$ as described above, then we have the estimate
\begin{equation}
\label{eq:graph-trace-global}
	\sum_{v_j \in \mathcal{V}_0} \deg v_j |f(v_j)|^2 \leq 2\|f\|_{\G_0}\|f'\|_{\G_0} + \frac{1}{\ell_{\G_0}} \|f\|_{\G_0}^2.
\end{equation}
\end{lemma}

For the proof, we will use the following cut-off functions; for each $v_j \in \mathcal{V}$, we define $\varphi_j \in H^1 (\G)$ with support in $\Scal_j^\xi$ by setting
\begin{equation}
		\label{eq:cut-off-function}
	\varphi_j^\xi (x) = \begin{cases}  1 - \frac{\dist (x,v_j)}{\xi(e)\ell_e} \qquad &\text{if } x \in \xi(e)e \subset \Scal_j^\xi\\ 0 \qquad &\text{otherwise}.\end{cases}
\end{equation}
Then clearly $0 \leq \varphi_j^\xi \leq 1$; moreover, since we are assuming that $\G$ does not have any loops, if $\xi(e)=1$ for all $e$ then the collection $(\varphi_j^\xi)_{j=1}^n$ is a partition of unity.

\begin{proof}[Proof of Lemma~\ref{lem:graph-trace-estimate}]
Let $v_j \in \mathcal{V}$ be arbitrary, let $\Scal_j^\xi$ and $\varphi_j^\xi$ be as described above, and let $f \in H^1(\G)$ be arbitrary. Then for each edge $\xi(e)e$ of $\Scal_j^\xi$ the fundamental theorem of calculus on that interval applied to the function $|f|^2\varphi_j^\xi$ plus the fact that $\varphi_j^\xi(0)=0$ gives
\begin{displaymath}
	|f(v_j)|^2 = \int_0^{\xi(e)\ell_e} (|f|^2\varphi_j^\xi)'\,\dx = \int_0^{\xi(e)\ell_e} 2 \varphi_j^\xi \Re (\bar{f}\,f') + |f|^2(\varphi_j^\xi)'\,\dx,
\end{displaymath}
and summing over all edges $\xi(e)e\sim v_j$ yields
\begin{align*}
	\deg v_j |f(v_j)|^2 &= \int_{\Scal_j^\xi} 2 \varphi_j^\xi \Re (\bar{f}\,f') + |f|^2\varphi_j'\,\dx \\
						&\leq 2\|\varphi_j^\xi\|_{\Scal_j^\xi,\infty}\|f\|_{\Scal_j^\xi}\|f'\|_{\Scal_j^\xi} + \|(\varphi_j^\xi)'\|_{\Scal_j^\xi,\infty}\|f\|_{\Scal_j^\xi}^2.
\end{align*}
Using that $\|\varphi_j^{\xi}\|_\infty = 1$ and $\|(\varphi_j^{\xi}){'}\|_\infty = 1/\ell_{\Scal_j^\xi}$ yields \eqref{eq:graph-trace-local}. For \eqref{eq:graph-trace-global}, we argue similarly but distinguish edges which are incident with two vertices of $\mathcal{V}_0$. More precisely, if $v_i,v_j \in \mathcal{V}_0$ are two distinct vertices and $v_i\sim e\sim v_j$, then we write $\Scal_j=\Scal_j^1$ for the trivial scaling factor $\xi=1$ (as well as $\phi_j=\phi_j^1$) and obtain the estimate
\begin{displaymath}
	|f(v_i)|^2+|f(v_j)|^2 = \int_0^{\ell_e} 2 \varphi_i \Re (\bar{f}\,f') + |f|^2\varphi_i' + 2 \varphi_j \Re (\bar{f}\,f') + |f|^2\varphi_j'\,\dx.
\end{displaymath}
But since $\varphi_i=1-\varphi_j$ and $\varphi_i' = -\varphi_j'$ on $e$, this reduces to
\begin{displaymath}
	|f(v_i)|^2+|f(v_j)|^2\leq 2\|f\|_{e}\|f'\|_{e}.
\end{displaymath}
We now sum over all edges $e \subset \G_0 = \bigcup_{v_j \in \mathcal{V}_0} \Scal_j$ both of whose incident vertices are in $\mathcal{V}_0$. To these we also sum the estimates
\begin{displaymath}
	|f(v_i)|^2 \leq 2\|f\|_{e}\|f'\|_{e} + \frac{1}{\ell_e}\|f\|_e^2,
\end{displaymath}
as obtained above, over all edges $e$ in $\mathcal{E}_0$ which have only one incident vertex $v_i$ in $\mathcal{V}_0$. Since each edge in the union $\G_0$ of the spanning stars of $\mathcal{V}_0$ is counted only once, this yields
\begin{displaymath}
	\sum_{v_j \in \mathcal{V}_0} \deg v_j|f(v_j)|^2 \leq 2\|f\|_{\G_0}\|f'\|_{\G_0} + \frac{1}{\ell_{\G_0}} \|f\|_{\G_0}^2,
\end{displaymath}
that is, \eqref{eq:graph-trace-global}.
\end{proof}

We can now give the proofs of Theorems~\ref{thm:numerical-range} and~\ref{thm:imaginary-control}.

\begin{proof}[Proof of Theorem~\ref{thm:numerical-range}]
(1) Let
\begin{displaymath}
	\lambda = \|f'\|_{\G}^2 + \sum_{j=1}^k\alpha_j|f(v_j)|^2,
\end{displaymath}
$f \in H^1(\G)$, $\|f\|_{\G}=1$, be any point in $W(a_\alpha)$. If we set $t:=\|f'\|_{\G}^2$, $s_j:=|f(v_j)|^2$ for each $v_j\in\V_R$ and we consider $\Scal_j^\xi$ for $\xi(e)=1/2$ for each $e\sim v_j$, then the stars $\Scal_j^\xi$ are all pairwise disjoint, $j=1,\ldots,k$. Then $\lambda$ has the form $\lambda = t + \sum_{j=1}^k \alpha_j s_j$, and the first statement of Lemma~\ref{lem:graph-trace-estimate} (together with the estimate that every star $\Scal_j^{1/2}$ has length at least $\ell_\G/2$) yields $s_j\leq \frac{2}{\MD}\sqrt{\tau_j}+\frac{2}{\MD\ell_\G}$ for each $j=1,\dots,k$, where the $\tau_j = \|f\|_{\Scal_j^{1/2}}^2$ are as in the statement of the theorem.

(2) Here we set $t:=\|f'\|_{\G}^2$ as before, but now $s:=\sum_{j=1}^k|f(v_j)|^2$, $\lambda$ has the form $\lambda=t+\alpha s$, and the estimate \eqref{eq:graph-trace-global} from Lemma~\ref{lem:graph-trace-estimate} implies that $s \leq \frac{2}{\MD}\sqrt{t} + \frac{1}{\MD\ell_\G}$.
\end{proof}

\begin{proof}[Proof of Theorem~\ref{thm:imaginary-control}]
(1) We simply note that, if $f \in H^1(\G)$ is an eigenfunction corresponding to $\lambda$, normalised so that $\|f\|_{\G}=1$, then by the second statement of Lemma~\ref{lem:graph-trace-estimate} applied to the union $\G_0$ of the stars $\Scal_j^1$, $j=1,\ldots,k$, whose total length we estimate from below by $\ell_\G$,
\begin{displaymath}
\begin{split}
	|\Im \lambda| 
	    = \left|\sum_{j=1}^k\Im \alpha_j|f(v_j)|^2\right| 
	&\leq \sum_{j=1}^k \frac{|\Im\alpha_j|}{\deg v_j} \deg v_j|f(v_j)|^2 \\
	&\leq \max_{j=1,\dots,k} \frac{|\Im\alpha_j|}{\deg v_j}\left[2\|f'\|_{\G}+\frac{1}{\MD\ell_\G}\right]\\
	&\leq \max_{j=1,\dots,k} \frac{|\Im\alpha_j|}{\deg v_j}\left[2\sqrt{\Re\lambda}+\frac{1}{\MD\ell_\G}\right],
\end{split}
\end{displaymath}
where the last inequality follows from taking the real part of the quadratic form \eqref{eq:qg-form} for $\lambda\in\C$ since $\Re \alpha$ was assumed non-negative. 

(2) We use the following weighted trace inequality: fix $k'\in\{1,\dots,k\}$ such that (after relabelling the $v_1,\dots,v_k$ if necessary) $\Re \alpha_j< 0$ if and only if $j\leq k'$. Then for every $\delta>0$ there exists a constant $C=C(\G,\Re\alpha_1 , \ldots, \Re\alpha_{k'}, \delta)$ such that
\begin{displaymath}
	0 \leq \sum_{j=1}^{k'} (-\Re \alpha_j) |f(v_j)|^2 \leq \delta\|f'\|_{\G}^2 + C\|f\|_{\G}^2
\end{displaymath}
for all $f \in H^1(\G)$, which can be obtained from the usual trace inequality by a standard $\varepsilon$-$C(\varepsilon)$ argument. Since for the eigenfunction $f$, normalised so that $\|f\|_{\G}=1$,
\begin{displaymath}
	\Re \lambda = \|f'\|_{\G}^2 + \sum_{j=1}^k \Re \alpha_j |f(v_j)|^2 
		\geq \|f'\|_{\G}^2 -\delta\|f'\|_{\G} -C
\end{displaymath}
it follows that
\begin{displaymath}
	\|f'\|_{\G} \leq \frac{\sqrt{\Re \lambda + C}}{\sqrt{1-\delta}}.
\end{displaymath}
If we now write $1-\varepsilon$ for $1/\sqrt{1-\delta}$, then the same argument as in (1), viz.
\begin{displaymath}
	|\Im \lambda| = \left|\sum_{j=1}^k\Im \alpha_j|f(v_j)|^2\right| \leq
	\max_{j=1,\dots,k} \frac{|\Im\alpha_j|}{\deg v_j}\left[2\|f'\|_{\G}+\frac{1}{\MD\ell_\G}\right],
\end{displaymath}
leads to \eqref{eq:imaginary-control-negative}.
\end{proof}

\begin{proof}[Proof of Corollary~\ref{cor:est-real}]
This follows directly from Theorem~\ref{thm:numerical-range}(2); indeed, for any eigenvalue $\lambda$
\begin{displaymath}
	\Re \lambda = t + \Re \alpha \cdot s \geq t + \Re \alpha \left(\frac{2\sqrt{t}}{\MD} + \frac{1}{\MD\ell_\G}\right).
\end{displaymath}
A short calculation shows that the latter expression is minimised over all possible $t>0$ when $t=\alpha^2/\MD^2$; this then yields \eqref{eq:est-real}.
\end{proof}

We now prove the complementary upper estimates \eqref{eq:est-real-upper} in the case that $\alpha$ is real and negative and finish with the proof of Theorem~\ref{cor:qg-behaviour-neg-alpha}. Both will rely on the variational (min-max) characterisation of the eigenvalues, valid for all real $\alpha$ (see, e.g., \cite[Section~4.1]{BKKM}), as well as the following eigenvalue estimate for stars.

\begin{lemma}
\label{lem:upper-est-star}
Let $\Scal$ be a star with a Robin parameter of strength $\alpha$ at its central vertex of degree $\MD$ and Dirichlet conditions at all other vertices. Then its first eigenvalue $\lambda_1^D(\alpha,\Scal)$ satisfies
\begin{equation}
\label{eq:lem-dirichlet-star-bound}
	\lambda_1^D(\alpha,\mathcal{S}) \leq  -\left(\frac{\alpha}{\MD} + \frac{1}{\ell_\Scal}\right)^2 < 0
\end{equation}
if $\alpha<-\frac{\MD}{\ell_\Scal}$, where as before $\ell_{\Scal}$ denotes the length of the shortest edge of $\Scal$.
\end{lemma}

\begin{proof}[Proof of Lemma~\ref{lem:upper-est-star}]
Since Dirichlet conditions are imposed on all degree one vertices, given an arbitrary compact star $\Scal$ we can use domain monotonicity with respect to inclusion to reduce to the special case of an equilateral star: indeed, if $\tilde{\Scal}$ denotes the truncated star having all edges of length $\ell_\Scal$, then $\lambda_1^D(\alpha,\Scal)\leq\lambda_1^D(\alpha,\tilde{\Scal})$, and we may thus assume that $\Scal=\tilde{\Scal}$ is in fact an equilateral star where each edge has length $\ell_\Scal$. 
We observe that the secular equation for $-\lambda_1^D(\alpha,\mathcal{S}) > 0$ reads
\begin{equation}
\label{eq:secular-star}
	\sqrt{\lambda} \coth (\sqrt{\lambda}\ell_\Scal) = -\frac{\alpha}{\MD},
\end{equation}
that is, $-\lambda_1^D(\alpha,\mathcal{S})$ is the smallest solution $\lambda>0$ of this equation. This follows from a short calculation using the vertex conditions and the symmetry property that the eigenfunction must be invariant under permutations of the $\MD$ equal edges of $\mathcal{S}$ (cf., e.g., \cite[Section~5]{BKKM}). Now the elementary inequality
\begin{displaymath}
	\coth (x) \leq \frac{1}{x} + 1, \qquad x>0,
\end{displaymath}
applied to the left-hand side of \eqref{eq:secular-star} gives
\begin{displaymath}
	\frac{1}{\ell_\Scal} + \sqrt{\lambda} \geq - \frac{\alpha}{\MD}.
\end{displaymath}
This is nontrivial if and only if $\alpha < -\MD/\ell_\Scal$. In this case, rearranging gives \eqref{eq:lem-dirichlet-star-bound}.
\end{proof}

\begin{proof}[Proof of the upper bound in Remark~\ref{rem:est-real}]
The bound $\lambda_1 (\alpha,\G) \leq \alpha k/|\G|$ follows immediately from taking $f \equiv 1$, that is, the eigenfunction corresponding to $\alpha=0$, as a test function in the variational characterisation. The other estimate will follow immediately from Lemma~\ref{lem:upper-est-star} and the inequality
\begin{displaymath}
	\lambda_1 (\alpha,\G) \leq \lambda_1^D (\alpha,\Scal),
\end{displaymath}
where $\Scal = \Scal_1^1$ is the star subgraph of $\G$ with central vertex $v_1$ (which we recall has degree $\deg v_1 = \MD = \min_{j=1,\ldots,k} \deg v_j$), as introduced above. This inequality, in turn, follows since the eigenfunction associated with $\lambda_1^D(\alpha,\mathcal{S})$, extended by zero to the rest of $\G$, may thus be canonically identified with a function in $H^1(\G)$ whose Rayleigh quotient is exactly equal to $\lambda_1^D(\alpha,\mathcal{S})$; equivalently, we may appeal directly to \cite[Theorem~3.10(1)]{BKKM}.
\end{proof}

We finish with the proof of Theorem~\ref{cor:qg-behaviour-neg-alpha}.

\begin{proof}[Proof of Theorem~\ref{cor:qg-behaviour-neg-alpha}]
The existence of $k$ eigenvalues with the claimed asymptotics, and the fact that non-divergent eigenvalues converge to points in $\sigma(-\Delta_{\V_R}^D)$ follow immediately from Theorem~\ref{thm:qg-behaviour}. We next show that there are no more than $k$ divergent eigenvalues. This follows from a standard interlacing statement: denoting the $k$th eigenvalue of $-\Delta_{\V_R}^D$ (counted with multiplicities) by $\lambda_k^D$, since the forms associated with $-\Delta_{\V_R}^\alpha$ and $-\Delta_{\V_R}^D$ agree on the form domain $H^1_0 (\G,\V_R)$ of the latter, and the quotient space $H^1(\G) / H^1_0 (\G,\V_R)$ has dimension $k$, it follows from the min-max characterisation of the eigenvalues that
\begin{displaymath}
	\lambda_{j-k}^D \leq \lambda_j(\alpha) \leq \lambda_j^D
\end{displaymath}
for all $\alpha \in \R$ and all $j \geq k + 1$ (see also, e.g., \cite[Section~3.1.6]{BerkolaikoKuchment} or \cite[Sections~3.1 and~4.1]{BKKM}). Hence $\lambda_j (\alpha)$ remains bounded whenever $j\geq k+1$, and so by Corollary~\ref{cor:dno-qg-asymptotics-duality} converges to an eigenvalue of the Dirichlet Laplacian.

It remains to prove that for $\alpha < -2\max\limits_{j=1,\dots,k}\left\{\frac{\deg v_j}{\ell_j}\right\}$ the Robin Laplacian has exactly $k$ negative eigenvalues: by the above reasoning, it suffices to find one fixed $\alpha$ for which it has at least $k$ such negative eigenvalues. To this end, for each $j=1,\ldots,k$, we consider each star $\Scal_j^{1/2}$ subgraph of $\G$ with Robin condition at its central vertex $v_j$; denote by $\psi_j$ the test function equal to the eigenfunction for $\lambda_1^D(\alpha,\Scal_j^{1/2})$ on $\Scal_j^{1/2}$, extended by zero to a function in $H^1(\G)$, and whose Rayleigh quotient equals $\lambda_1^D (\alpha,\Scal_j^{1/2})$. Then, since the supports of $\psi_j$ are pairwise disjoint, we can define the $k$-dimensional space $\mathcal{H}_k:=\bigoplus_{j=1}^k \psi_j \subset H^1(\G)$ as a space of test functions for $\lambda_k(\alpha,\G)$. If we choose any
\begin{displaymath}
\alpha < -2\max_{j=1,\dots,k}\left\{\frac{\deg v_j}{\ell_{\Scal_j}}\right\},
\end{displaymath}
then by Lemma~\ref{lem:upper-est-star} each function $\psi_j$, and thus every function in $\mathcal{H}_k$ has a negative Rayleigh quotient (where one should not forget the scaling factor $2\ell_{\Scal_j}^{1/2}=\ell_{\Scal_j}$). It follows from the min-max characterisation that $\lambda_k (\alpha,\G) < 0$ for such $\alpha$.
\end{proof}

\end{document}